\documentclass[12pt, reqno, twoside, letterpaper]{amsart}

\usepackage{paperstyle}
\usepackage{mathtools}
\usepackage{bm}
\usepackage{upgreek}

\def\sLT{\EuScript{L}_T}
\def\sLTtau{\EuScript{L}_{\natural}}
\def\sLx{\EuScript{L}_x}

\def\bg{\bm{\upgamma}}
\def\bgp{\bm{\upgamma'}}

\def\obg{\bm{\Sigma\upgamma}}
\def\obgp{\bm{\Sigma\upgamma'}}
\def\eobg{\bm{\Sigma\upgamma}}
\def\eobgp{\bm{\Sigma\upgamma'}}

\title[Pair correlation for sums of two ordinates]
{Pair correlation for sums of two ordinates\\ of zeros of the Riemann zeta function}

\author[W.\ D.\ Banks]{William D.\ Banks}

\address{Department of Mathematics,
         University of Missouri,
         Columbia MO, USA.}

\email{bankswd@missouri.edu}

\date{\today}

\begin{document}

\begin{abstract}
Assuming the Riemann Hypothesis, we extend Montgomery's pair correlation
method to study the distribution of differences between sums
$\gamma_1+\gamma_2$ of two ordinates of nontrivial zeros of the Riemann
zeta function. For the associated pair correlation function
we prove that
\[
G_2(\alpha,T)=\biggl\{\frac{\log T}{T^{2\alpha}}+
\frac{4\alpha^3}{3T^\alpha}\biggr\}
\biggl\{1+O\(\frac{1}{\log\log T}\)\biggr\}
\]
uniformly for $0\le\alpha\le \frac23-\frac{2\log\log T}{\log T}$.
In contrast with the conjectured GUE statistics of the ordinates 
themselves, this result points to an absence of level repulsion
among sums of two ordinates, the two-point correlation function
of the sums being identically one, as for a Poisson process.
\end{abstract}

\maketitle


{\Large\section{Introduction}\label{sec:intro}}

\subsection{Background}\label{sec:background}
Over a half-century ago, in his celebrated pair correlation paper,
Montgomery \cite{Mont} studied the distribution of the differences
$\gamma-\gamma'$ between ordinates of zeros $\rho=\tfrac12+i\gamma$
of the Riemann zeta function $\zeta(s)$,
assuming the Riemann hypothesis~(RH).
His \emph{pair correlation conjecture} (PCC) asserts that
$1-((\sin\pi u)/\pi u)^2$ is the two-point correlation function for
zeros of $\zeta(s)$ on the critical line. In other words, for any
real numbers $\alpha<\beta$, PCC predicts
\[
\lim\limits_{T\to\infty}
\frac{\bigl|\bigl\{(\gamma,\gamma'):\gamma,\gamma'\in(0,T],\frac{2\pi\alpha}{\log T}<\gamma-\gamma'\le\frac{2\pi\beta}{\log T}\bigr\}\bigr|\vphantom{\Big|}}{\frac{1}{2\pi}T\log T}
=\int_\alpha^\beta\biggl\{1-\(\frac{\sin\pi u}{\pi u}\)^2\biggr\}\dd u
+\delta(\alpha,\beta),
\]
where zeros are counted with multiplicity, and
$\delta(\alpha,\beta)\defeq 1$ if $0\in[\alpha,\beta]$,
$\delta(\alpha,\beta)\defeq 0$ otherwise; the $\delta$-term accounts
for the diagonal pairs $\gamma=\gamma'$.
For general background on the conjecture and its implications,
we refer the reader to the excellent survey article by
Goldston \cite{Goldston} and to the references therein.
Higher correlations of the zeros were later studied by
Hejhal \cite{Hejhal} and by Rudnick and Sarnak \cite{RudSar}.
We note that the correlation sums considered here involve quadruples
of ordinates only through the degenerate linear form
$\gamma_1+\gamma_2-\gamma_1'-\gamma_2'$; thus,
the associated test functions lie outside the admissible
class of \cite{RudSar}, and so our results are
not contained in the theory of $n$-level
correlations.

To access the distribution of the differences $\gamma-\gamma'$,
Montgomery studied the pair correlation function
\[
F(\alpha)=F(\alpha,T)\defeq\frac{2\pi}{T\log T}
\ssum{\gamma,\gamma'\in(0,T]}T^{i\alpha(\gamma-\gamma')}w(\gamma-\gamma'),
\]
where $\alpha$ and $T\ge 2$ are real, and $w$ is the weight
\be\label{eq:w(u)defn}
w(u)\defeq\frac{4}{u^2+4}\qquad(u\in\R).
\ee

\begin{theorem*}[Montgomery \cite{Mont}]
For any fixed $\eps>0$, one has
\[
F(\alpha)=\bigl(1+o(1)\bigr)T^{-2\alpha}\log T+\alpha+o(1)
\]
uniformly for $0\le\alpha\le 1-\eps$.
\end{theorem*}

\subsection{Statement of results}\label{sec:results}
In the present note, we extend Montgomery's approach
to study a pair correlation function for sums of two ordinates.

\begin{definition*}
Assume \text{\rm RH}. Let $\fZ_2(T)$ be
the multiset  of all ordered pairs
$\bg=(\gamma_1,\gamma_2)$ such that $\gamma_j>0$
and $\zeta(\tfrac12+i\gamma_j)=0$ for each $j$, and
$\gamma_1+\gamma_2\le T$. Each
pair $(\gamma_1,\gamma_2)$ occurs $m_1m_2$ times in
$\fZ_2(T)$, where $m_j$ is the multiplicity
of the zero $\tfrac12+i\gamma_j$.
\end{definition*}

\noindent Our main object of study
is the normalized pair correlation function given by
\be\label{eq:G2alphaT-defn}
G_2(\alpha)=G_2(\alpha,T)\defeq
3\(\frac{2\pi}{T\log T}\)^{3}
\ssum{\bg,\bgp\in\fZ_2(T)}
T^{i\alpha(\eobg-\eobgp)}
w(\obg-\obgp),
\ee
where here and in the sequel we use the notation
\[
\obg\defeq\gamma_1+\gamma_2\qquad
\text{for each}\quad\bg=(\gamma_1,\gamma_2).
\]
The factor $3$ in \eqref{eq:G2alphaT-defn} reflects the linearly
increasing density of the multiset
\[
\{\obg:\bg\in\fZ_2(T)\}.
\]
Since $|\fZ_2(u)|\sim\tfrac12\(\frac{u\log u}{2\pi}\)^2$, we have
\[
\frac13\(\frac{T\log T}{2\pi}\)^3
\sim\frac{2\pi}{\log T}\int_0^T\(\frac{u\log^2 u}{4\pi^2}\)^2\dd u,
\]
the natural pair-count normalization for a point process of
nonconstant intensity.
As with Montgomery's $F$, the function $G_2$ is real and even, and
$G_2(\alpha)\ge 0$ for all $\alpha$. Indeed, since
$w(v)=\int_\R e^{-2|u|}e^{ivu}\dd u$, the sum in
\eqref{eq:G2alphaT-defn} is equal to
\[
\int_\R e^{-2|u|}\,\biggl|\ssum{\bg\in\fZ_2(T)}
T^{i\alpha\eobg}e^{i\eobg u}\biggr|^2\dd u.
\]
Our main result is the following theorem,
which is proved in \S\ref{sec:proof}.

\begin{theorem}\label{thm:main}
Assume \text{\rm RH}, and let $T\ge 100$. We have
\[
G_2(\alpha)=\biggl\{\frac{\log T}{T^{2\alpha}}
+\frac{4\alpha^3}{3T^\alpha}\biggr\}
\biggl\{1+O\(\frac{1}{\log\log T}\)\biggr\}
\]
uniformly for $0\le\alpha\le \frac23-\frac{2\log\log T}{\log T}$.\end{theorem}

Sums over the differences $\obg-\obgp$ can be evaluated via
Theorem~\ref{thm:main}. If $r\in L^1(\R)$ and
$\hat r(\alpha)\defeq\int_\R r(u)e^{2\pi i\alpha u}\dd u$
also lies in $L^1(\R)$, then by Fourier inversion and the evenness
of $G_2$ we have, as in \cite[eq.\,(3)]{Mont}:
\[
\ssum{\bg,\bgp\in\fZ_2(T)}r\!\((\obg-\obgp)\frac{\log T}{2\pi}\)
w(\obg-\obgp)=\frac13\(\frac{T\log T}{2\pi}\)^{3}\int_\R
G_2(\alpha)\hat r(\alpha)\dd \alpha.
\]
For example (cf.\ \cite[Cor.\,1]{Mont}), if $0<\lambda<\tfrac23$, then
\dalign{
\ssum{\bg,\bgp\in\fZ_2(T)}
\(\frac{\sin\lambda(\obg-\obgp)\log T}
{\lambda(\obg-\obgp)\log T}\)
w(\obg-\obgp)
&\sim\frac{1}{6\lambda}\(\frac{T\log T}{2\pi}\)^3,\\
\ssum{\bg,\bgp\in\fZ_2(T)}
\(\frac{\sin\tfrac12\lambda(\obg-\obgp)\log T}
{\tfrac12\lambda(\obg-\obgp)\log T}\)^2
w(\obg-\obgp)
&\sim\frac{1}{3\lambda}\(\frac{T\log T}{2\pi}\)^3.
}

We now describe the statistical meaning of
Theorem~\ref{thm:main}. Montgomery's theorem, coupled with the PCC,
predicts that the rescaled differences
$(\gamma-\gamma')\frac{\log T}{2\pi}$
are distributed according to GUE statistics, which exhibit the
characteristic quadratic level repulsion of the sine kernel near the
origin. Theorem~\ref{thm:main} tells a very different story for sums
of two ordinates. Indeed, the term $T^{-2\alpha}\log T$ carries unit
mass concentrated near $\alpha=0$ and corresponds, on the Fourier
side, to the constant density one, whereas the term
$\tfrac43\alpha^3T^{-\alpha}$ contributes only $O(\log^{-4}T)$.
Moreover, in contrast with the situation for $F(\alpha)$, no constant
term is present, since the diagonal contribution to $G_2(\alpha)$ is
$O((T\log T)^{-1})$. Thus, if one assumes (in analogy with
Montgomery's conjecture that $F(\alpha)=1+o(1)$ for $\alpha\ge1$, and
as Theorem~\ref{thm:main} strongly suggests) that
$G_2(\alpha)=o(1)$ holds uniformly for $1\le\alpha\le A$ with any
fixed $A>1$, then for fixed $\alpha<\beta$ one is led to the
prediction that
\[
\lim\limits_{T\to\infty}
\frac{\bigl|\bigl\{(\bg,\bgp):\bg,\bgp\in\fZ_2(T),
 \frac{2\pi\alpha}{\log T}<\obg-\obgp\le\frac{2\pi\beta}{\log T}\bigr\}\bigr|
 \vphantom{\Big|}}{\tfrac13\(\frac{1}{2\pi}T\log T\)^{3}}
=\int_\alpha^\beta 1\dd u=\beta-\alpha.
\]

We prove Theorem~\ref{thm:main} in the range
$0\le\alpha\le\frac23-\frac{2\log\log T}{\log T}$. Throughout this range
the first term $T^{-2\alpha}\sLT$ dominates, so the asymptotic already
determines the density-one behavior underlying the above prediction. The
extension to $\alpha\to 1^-$, where the crossover to the second term
occurs, relies on a cancellation estimate for one-sided sums over
pairs of ordinates that we do not address here.

The prediction suggests that 
the differences of sums of two ordinates exhibit no
repulsion at all, the two-point correlation function being
identically \emph{one}. Thus, the ordinates themselves
are expected to mimic the eigenvalues of a random Hermitian matrix,
whereas their pairwise sums behave like a Poisson process, and the
arithmetic structure detected by the sine kernel is entirely washed
out by a single summation.

Sums over unnormalized differences
$\gamma_1+\gamma_2-\gamma_3-\gamma_4$ were studied by
Yudelevich~\cite{Yudelevich}, who adapted the method of Ford and
Zaharescu \cite{FordZah} to show, unconditionally, that
\[
\ssum{\gamma_1,\gamma_2,\gamma_3,\gamma_4\in(0,T]}
h(\gamma_1+\gamma_2-\gamma_3-\gamma_4)
=\frac{T^3}{24\pi^4}\int_\R
h(t)\bigl(K_4(2+it)+K_4(2-it)\bigr)\dd t
+O\(\frac{T^3}{(\log T)^{1/4}}\)
\]
holds for any fixed $h\in L^1(\R)$ of mean zero satisfying mild decay
conditions, where $K_4(s)\defeq\sum_n\Lambda(n)^4n^{-s}$ with
$\Lambda$ the von Mangoldt function. As $h$ is
fixed, this determines the distribution of the linear form only at
bounded frequencies, that is, for $\alpha\ll(\log T)^{-1}$ in the
normalization of \eqref{eq:G2alphaT-defn}. Theorem~\ref{thm:main}
cannot be derived from this result, since the hypothesis that $h$ has
mean zero removes the mass near $\alpha=0$, and since for fixed
$\alpha>0$ the error term above exceeds
$\(\frac{1}{2\pi}T\log T\)^3G_2(\alpha)$ in size. The series $K_4$
nevertheless resurfaces in our work as the diagonal of the $B$-integral
in \S\ref{sec:B2}. It is also worth mentioning that
the repulsion observed in
\cite{Yudelevich} is of a different nature from the absence of
repulsion described above. There, the values
$\gamma_1+\gamma_2-\gamma_3-\gamma_4$ tend to avoid the ordinates
$\gamma$ themselves, an effect reflected in the dips of
$K_4(2+it)+K_4(2-it)$ at $t=\gamma$. That phenomenon concerns the
interaction of the linear form with the zero spectrum, whereas our
Theorem~\ref{thm:main} focuses on the mutual spacings of the sums.

\subsection{Underlying approach}\label{sec:underlying}

We assume RH throughout.
To prove Theorem~\ref{thm:main}, we follow
Montgomery's method closely, using various techniques
to carry out individual steps.
For the benefit of the reader and to set the stage for our work, we
briefly review Montgomery's argument.

Let $x>1$, $x\ne p^k$. By an old result of Landau \cite{Landau},
the relation
\[
\sum_{n\le x}\Lambda(n)n^{-s}=-\frac{\zeta'}{\zeta}(s)
+\frac{x^{1-s}}{1-s}-\sum_\rho\frac{x^{\rho-s}}{\rho-s}
+\sum_{n\in\N}\frac{x^{-2n-s}}{2n+s}
\]
holds provided that $s\ne 1$, $s\ne\rho$, $s\ne -2n$.
Multiplying by $x^{s-1/2}$, we see that
\begin{alignat}{2}
\label{eq:ice1}
\sum_\rho\frac{x^{\rho-1/2}}{\rho-s}
&=-x^{-1/2}\sum_{n\le x}\Lambda(n)\(\frac{x}{n}\)^s
-\frac{\zeta'}{\zeta}(s)x^{s-1/2}+\frac{x^{1/2}}{1-s}
+\sum_{n\in\N}\frac{x^{-2n-1/2}}{2n+s},\\
\label{eq:ice2}
\sum_\rho\frac{x^{\rho-1/2}}{\rho-s}
&=x^{-1/2}\sum_{n>x}\Lambda(n)\(\frac{x}{n}\)^s
+\frac{x^{1/2}}{1-s}+\sum_{n\in\N}\frac{x^{-2n-1/2}}{2n+s}
\qquad(\sigma>1).
\end{alignat}
As the individual terms in these relations are continuous for all
$x\ge 1$, we no longer exclude $x=1$ or $x=p^k$ from our considerations.

Next, we specialize.
Taking $s=-\tfrac12+it$ in \eqref{eq:ice1}, and noting that
every nontrivial zero of $\zeta(s)$ has the form
$\rho=\tfrac12+i\gamma$ under RH, we get that
\dalign{
&\sum_\gamma\frac{x^{i\gamma}}{1-i(t-\gamma)}
=-\frac{\zeta'}{\zeta}(-\tfrac12+it)x^{-1+it}\\
&\qquad\qquad
-x^{-1/2}\sum_{n\le x}\Lambda(n)\(\frac{x}{n}\)^{-1/2+it}
+\frac{x^{1/2}}{\tfrac32-it}
+\sum_{n\in\N}\frac{x^{-2n-1/2}}{2n-\tfrac12+it}.
}
Similarly, taking $s=\tfrac32+it$ in \eqref{eq:ice2} and
multiplying by $-1$, we have
\[
\sum_\gamma\frac{x^{i\gamma}}{1+i(t-\gamma)}
=-x^{-1/2}\sum_{n>x}\Lambda(n)\(\frac{x}{n}\)^{3/2+it}
+\frac{x^{1/2}}{\tfrac12+it}
-\sum_{n\in\N}\frac{x^{-2n-1/2}}{2n+\tfrac32+it}.
\]
Summing these relations, it follows that
\be\label{eq:mont-orig}
\sum_\gamma x^{i\gamma}W_\gamma(t)=A(t)+B(t)+C(t)+D(t),
\ee
where
\dalign{
W_\gamma(t)&\defeq\frac{2}{(t-\gamma)^2+1},\\
A(t)&\defeq-\frac{\zeta'}{\zeta}(-\tfrac12+it)\,x^{-1+it},\\
B(t)&\defeq -x^{-1/2}\biggl\{
\sum_{n\le x}\Lambda(n)\(\frac{x}{n}\)^{-1/2+it}
+\sum_{n>x}\Lambda(n)\(\frac{x}{n}\)^{3/2+it}\biggr\},\\
C(t)&\defeq\frac{x^{1/2}}{\tfrac32-it}+\frac{x^{1/2}}{\tfrac12+it},\\
D(t)&\defeq\sum_{n\in\N}\frac{x^{-2n-1/2}}{2n-\tfrac12+it}
-\sum_{n\in\N}\frac{x^{-2n-1/2}}{2n+\tfrac32+it}.
}
Note that \eqref{eq:mont-orig} is the special case $\sigma=\tfrac32$
of \cite[eq.\,(22)]{Mont}. Using \eqref{eq:mont-orig}, we have
\be\label{eq:mont-norm-eq}
\int_0^T\biggl|\sum_\gamma x^{i\gamma}W_\gamma(t)\biggr|^2\dd t
=\int_0^T\bigl|A(t)+B(t)+C(t)+D(t)\bigr|^2\dd t.
\ee
Montgomery \cite{Mont} shows that the integral on
the left side of \eqref{eq:mont-norm-eq} is equal to
\[
2\pi\hskip-8pt\ssum{\gamma,\gamma'\in(0,T]}x^{i(\gamma-\gamma')}
w(\gamma-\gamma')+O(\log^3T)
=F(\alpha)\,T\log T+O(\log^3T),
\]
where $\alpha\defeq(\log x)/\log T$. On the other hand,
the integral on the right side \eqref{eq:mont-norm-eq}
can be estimated effectively using individual estimates for
\[
\int_0^T\bigl|A(t)\bigr|^2\dd t,\qquad
\int_0^T\bigl|B(t)\bigr|^2\dd t,\qquad
\int_0^T\bigl|C(t)\bigr|^2\dd t,\qquad
\int_0^T\bigl|D(t)\bigr|^2\dd t.
\]
Montgomery's theorem (see \S\ref{sec:background})
follows by comparing these estimates. In \S\ref{sec:proof},
we adapt this argument to handle sums of two ordinates.

{\Large\section{Preliminaries}\label{sec:prelims}}

Following Riemann, we use $s$ for a complex variable
with $\sigma\defeq\Re(s)$ and $t\defeq\Im(s)$. We write
$\tau=\tau(t)\defeq|t|+10$ for any given $t\in\R$. The parameters
$x\ge 1$ and $T\ge 2$ are real numbers, and we use the notation
\[
\sLx\defeq\log x\mand\sLT\defeq\log T,
\]
extensively in the sequel.
Any implied constants in the symbols $\ll$, $O$, $\asymp$, etc.,
are absolute unless dependence on
other parameters is specified explicitly.
As mentioned earlier, we assume RH throughout the note.

Our first lemma follows from the Cauchy-Schwarz
inequality; the proof is omitted.

\bigskip

\begin{lemma}\label{lem:C-S}
Let $T$ be large, let $F_1,\dots,F_N\colon\R\to\C$, and put
\[
I_j\defeq\int_0^T\bigl|F_j(t)\bigr|^2\dd t.
\]
If $I_1,I_2\ge I_3\ge\cdots\ge I_N$, then
\[
\int_0^T\biggl|\sum_{j=1}^NF_j(t)\biggr|^2\dd t
=I_1+I_2+2\Re\!\int_0^T\!F_1(t)\overline{F_2(t)}\dd t
+O\bigl(\sqrt{(I_1+I_2)\,I_3}\,\bigr),
\]
where the implied constant depends only on $N$.
\end{lemma}

The next lemma gives estimates for sums of $\Lambda(n)^k$
that are needed in \S\ref{sec:B2}.

\begin{lemma}\label{lem:wagyu-sirloin}
Fix an integer $k\ge 1$. All implied constants below depend only on $k$
and on the exponent $a$, and $\sLx$ denotes $\log x$.

\begin{enumerate}
  \item For any real $a>-1$,
  \[
  \sum_{n\le x}\Lambda(n)^k n^a
  =\frac{x^{a+1}\sLx^{\,k-1}}{a+1}
  \bigl\{1+O\bigl(\sLx^{-1}\bigr)\bigr\},
  \]
  and in the boundary case $a=-1$,
  \[
  \sum_{n\le x}\Lambda(n)^k n^{-1}
  =\frac{\sLx^{\,k}}{k}
  \bigl\{1+O\bigl(\sLx^{-1}\bigr)\bigr\}.
  \]

  \item For any real $a>1$,
  \[
  \sum_{n>x}\Lambda(n)^k n^{-a}
  =\frac{x^{1-a}\sLx^{\,k-1}}{a-1}
  \bigl\{1+O\bigl(\sLx^{-1}\bigr)\bigr\},
  \]
  and with an extra $\log n\,\log\log n$ weight,
  \[
  \sum_{n>x}\Lambda(n)^k n^{-a}\log n\,\log\log n
  =\frac{x^{1-a}\sLx^{\,k}\log\sLx}{a-1}
  \bigl\{1+O\bigl(\sLx^{-1}\bigr)\bigr\}.
  \]
\end{enumerate}
\end{lemma}

\begin{proof}[Proof sketch]
We prove the first estimate in (1); the remaining three are proved
in a similar way, the weight $\log n\,\log\log n$ in (2) contributing the
factor $\sLx\log\sLx$ since the tails are dominated by $n\asymp x$.

Put $f(u)\defeq u^{a}(\log u)^{k-1}$.
Since $\Lambda(n)^k$ and $\Lambda(n)(\log n)^{k-1}$ differ only on
proper prime powers, the sum in (1) equals
$\sum_{n\le x}\Lambda(n)f(n)$
up to an acceptable error. By partial summation, the latter sum is
\[
\int_{2^-}^{x}f(u)\dd \psi(u)
=\int_{2}^{x}f(u)\dd u+f(x)E(x)-\int_{2}^{x}f'(u)E(u)\dd u,
\]
where $E(u)\defeq\psi(u)-u\ll u\exp\bigl(-c\sqrt{\log u}\,\bigr)$ by the
prime number theorem. Both error terms are $O(x^{a+1}\sLx^{k-2})$.
Finally, integration by parts gives
\[
\int_{2}^{x}f(u)\dd u
=\frac{x^{a+1}\sLx^{k-1}}{a+1}
 -\frac{k-1}{a+1}\int_{2}^{x}\frac{f(u)}{\log u}\dd u+O(1)
=\frac{x^{a+1}\sLx^{k-1}}{a+1}
 \bigl\{1+O\bigl(\sLx^{-1}\bigr)\bigr\}.
\]
and the result follows.
\end{proof}

{\Large\section{Sums with ordinates}\label{sec:summing ords}}

In this section, we present results concerning sums over
the ordinates of zeros of $\zeta(s)$; these are needed for
our proof of Theorem~\ref{thm:main} in \S\ref{sec:proof}.

It is well known that the number $N(u)$
of nontrivial zeros $\rho=\tfrac12+i\gamma$ of $\zeta(s)$
with $0<\gamma\le u$ satisfies (under RH)
\be\label{eq:N(u)estimate}
N(u)=\frac{u}{2\pi}\log\frac{u}{2\pi\er}
+O\(\frac{\log u}{\log\log u}\)\qquad(u\ge 10);
\ee
see, for example, \cite[Cor.\,14.4]{MontVauBook}. Using
Riemann-Stieltjes integration, we deduce from~\eqref{eq:N(u)estimate}
the following uniform bounds for any $T_*\in[0,T]$:
\begin{alignat}{3}
\label{eq:shrink1}
\ssum{\gamma\\\gamma\not\in[0,T_*]}
\frac{1}{(t_*-\gamma)^2+1}&\ll\(\frac{1}{t_*+1}+\frac{1}{T_*-t_*+1}\)\sLT
\qquad&(t_*\in[0,T_*]),\\
\label{eq:shrink2}
\ssum{\gamma}\frac{1}{(t_*-\gamma)^2+1}&\ll \log(|t_*|+10)
\qquad&(t_*\in\R),\\
\label{eq:shrink3}
\ssum{\gamma\le T_*}
\frac{1}{(t_*-\gamma)^2+1}&\ll\frac{\sLT}{t_*-T_*+1}
\qquad&(t_*>T_*),\\
\label{eq:shrink4}
\ssum{\gamma\ge 0}
\frac{1}{(t_*-\gamma)^2+1}&\ll\frac{\log(|t_*|+10)}{|t_*|+1}
\qquad&(t_*\le 0),\\
\label{eq:shrink5}
\ssum{\gamma\ge 0}
\frac{1}{(t_*+\gamma)^2+1}&\ll\frac{\log(t_*+10)}{t_*+1}
\qquad&(t_*\ge 0).
\end{alignat}

The following lemma is used to estimate the $W$-integral
in \S\ref{sec:W2}.

\bigskip
\begin{lemma}\label{lem:lime}
For any $a,b\in(0,T]$, we have
\begin{align}
\label{eq:jack}
&\ssum{\gamma,\gamma'}\int_{\max\{a,b\}}^T\frac{1}{(t-a-\gamma)^2+1}
\cdot\frac{1}{(t-b-\gamma')^2+1}\dd t\\
\nonumber
&\qquad=\ssum{\gamma\in(0,T-a]\\\gamma'\in(0,T-b]}\
\int_\R\frac{1}{(t-a-\gamma)^2+1}
\cdot\frac{1}{(t-b-\gamma')^2+1}\dd t+O(\sLT^3),
\end{align}
where the sum on the left is over all pairs of ordinates
$\gamma,\gamma'$ $($of either sign$)$.
\end{lemma}

\begin{proof}
If $\max\{a,b\}\le t\le T$, then $t-a\in[0,T-a]$, so
\eqref{eq:shrink1} (with $t_*=t-a$, $T_*=T-a$) bounds the sum over
$\gamma\notin[0,T-a]$, while \eqref{eq:shrink2} handles the sum over
$\gamma'$; integrating in $t$ gives
\[
\ssum{\gamma,\gamma'\\\gamma\not\in[0,T-a]}
\int_{\max\{a,b\}}^T\frac{1}{(t-a-\gamma)^2+1}
\cdot\frac{1}{(t-b-\gamma')^2+1}\dd t\ll\sLT^3,
\]
and the same bound holds with the roles of $(\gamma,a)$ and
$(\gamma',b)$ interchanged. Thus, up to $O(\sLT^3)$, the sum in
\eqref{eq:jack} may be restricted to $\gamma\in(0,T-a]$ and
$\gamma'\in(0,T-b]$.

It remains to extend the $t$-integral from $[\max\{a,b\},T]$ to
all of $\R$. To do this, we bound the three complementary
ranges $(T,\infty)$, $(-\infty,\min\{a,b\}]$,
and $[\min\{a,b\},\max\{a,b\}]$, each by $O(\sLT^3)$.

For any $t>T$, we have $t-a>T-a$ and $t-b>T-b$;
applying \eqref{eq:shrink3} twice, 
\[
\ssum{\gamma\in(0,T-a]\\\gamma'\in(0,T-b]}
\int_T^\infty \frac{1}{(t-a-\gamma)^2+1}
\cdot\frac{1}{(t-b-\gamma')^2+1}\dd t
\ll\sLT^2\int_T^\infty\frac{dt}{(t-T+1)^2}\ll\sLT^2.
\]
Similarly, using \eqref{eq:shrink4}, we have
\dalign{
&\ssum{\gamma\in(0,T-a]\\\gamma'\in(0,T-b]}
\int_{-\infty}^{\min\{a,b\}}
\frac{1}{(t-a-\gamma)^2+1}
\cdot\frac{1}{(t-b-\gamma')^2+1}\dd t\\
&\qquad\qquad\ll\int_{-\infty}^{\min\{a,b\}}
\frac{\log(a-t+10)}{a-t+1}
\cdot\frac{\log(b-t+10)}{b-t+1}\dd t\ll\sLT^2.
}
Finally, we bound
\[
\ssum{\gamma\in(0,T-a]\\\gamma'\in(0,T-b]}
\int_{\min\{a,b\}}^{\max\{a,b\}}
\frac{1}{(t-a-\gamma)^2+1}
\cdot\frac{1}{(t-b-\gamma')^2+1}\dd t.
\]
Assuming (without loss of generality) that $a\ge b$, this is
\[
\ssum{\gamma\in(0,T-a]\\\gamma'\in(0,T-b]}
\int_b^a\frac{1}{(a-t+\gamma)^2+1}
\cdot\frac{1}{(t-b-\gamma')^2+1}\dd t.
\]
Using \eqref{eq:shrink5} for the sum over $\gamma$ and
\eqref{eq:shrink2} for the sum over $\gamma'$, this is
\[
\ll\sLT\int_b^a\frac{\log(a-t+10)}{a-t+1}\dd t\ll\sLT^3,
\]
and the proof is finished.
\end{proof}

The next lemma is needed for our estimate of the $A$-integral
in \S\ref{sec:A2}.

\begin{lemma}\label{lem:specialized}
Uniformly for $t\in\R$, we have
\[
\sum_{\gamma\in(0,T]}\log\tau_{t-\gamma}
=\frac{T\sLT\sLTtau}{2\pi}
+O(T\sLTtau\log\sLT)\qquad(T\ge 14),
\]
where $\tau_t\defeq\tau(t)=|t|+10$
and $\sLTtau\defeq\log\max\{T,\tau_t\}$.
\end{lemma}

\begin{proof}
Suppose first that $|t|>100T$. For any $\gamma\in(0,T]$ we have
$|t-\gamma|\ge|t|-T$, hence
\[
\tau_{t-\gamma}\asymp\tau_t
\mand
\log\tau_{t-\gamma}=\log\tau_t+O(1)=\sLTtau+O(1)
\]
since $\max\{T,\tau_t\}=\tau_t$ holds in this case. Consequently,
\[
\sum_{\gamma\in(0,T]}\log\tau_{t-\gamma}
=(\sLTtau+O(1))N(T),
\]
and the result follows from \eqref{eq:N(u)estimate}.

From now on, we assume $|t|\le 100T$.
Using Riemann--Stieltjes integration, we have
\[
\sum_{\gamma\in(0,T]}\log\tau_{t-\gamma}
=\int_{14}^T\log(|t-u|+10)\dd N(u).
\]
Put
\[
\cJ\defeq\{u\in[14,T]:|t-u|<T/\sLT\},
\]
which is an interval of length $O(T/\sLT)$.
For all $u\in[14,T]\setminus\cJ$, we have
\[
T/\sLT\le |t-u|+10\ll T,
\]
the upper bound using $|t|\le 100T$. Therefore,
\[
\log(|t-u|+10)=\sLT+O(\log\sLT)
\qquad(u\in[14,T]\setminus\cJ).
\]
It follows that
\dalign{
\int_{14}^T\log(|t-u|+10)\dd N(u)
&=\bigl\{\sLT+O(\log\sLT)\bigr\}N(T)
+O\(\sLT\int_\cJ\dd N(u)\)\\
&=\frac{T\sLT^{2}}{2\pi}
+O(T\sLT\log\sLT),
}
where we have used \eqref{eq:N(u)estimate} in the second step,
together with the bound $\int_\cJ dN(u)\ll T$, which follows from
\eqref{eq:N(u)estimate} and the length of $\cJ$.
Since $\sLTtau=\sLT+O(1)$ in this case, the lemma follows.
\end{proof}

The following result, which underlies our estimate of the
$B$-integral in \S\ref{sec:B2}, is a well known theorem
of Gonek \cite[Thm.\,1]{Gonek}. Under RH, for each integer $n\ge 2$
and all $t\ge 14$, one has
\be\label{eq:STLrocks}
S(n,t)\defeq\sum_{\gamma\in(0,t]}n^{i\gamma}=\lambda_nt+E(n,t),
\ee
where
\be\label{eq:lam-cE-defn}
\lambda_n\defeq-\frac{\Lambda(n)}{2\pi n^{1/2}},\qquad
E(n,t)\ll\cE(n,t)\defeq\begin{cases}
n^{1/2}\log t\,\log\log t
&\quad\hbox{if $n\le t$},\\
n^{1/2}\log n\,\log\log n
&\quad\hbox{if $n>t$}.\\
\end{cases}
\ee
Gonek's theorem holds unconditionally and is even more precise. We
require only the pointwise bound \eqref{eq:lam-cE-defn}, which in
\S\ref{sec:B2} controls the error term $R$ of the $B$-integral.

Finally, we record a simple estimate needed for the
$C$- and $D$-integrals in \S\ref{sec:C2+D2}.

\begin{lemma}\label{lem:harmonic}
For all $T\ge 20$, we have
\[
\sum_{\gamma\in(0,T]}\frac{1}{\gamma}\asymp\sLT^2.
\]
\end{lemma}

\begin{proof}
Since the least positive ordinate is $\gamma_1=14.13\cdots$,
partial summation gives
\[
\sum_{\gamma\in(0,T]}\frac{1}{\gamma}
=\frac{N(T)}{T}+\int_{14}^T\frac{N(u)}{u^2}\dd u,
\]
and the result follows from \eqref{eq:N(u)estimate}.
\end{proof}

{\Large\section{Proof of Theorem \ref{thm:main}}\label{sec:proof}}

\subsection{Basic identity}
For the convenience of the reader, we recall the definition:

\begin{definition*}
Assume \text{\rm RH}. Let $\fZ_2(T)$ be
the multiset  of all ordered pairs
$\bg=(\gamma_1,\gamma_2)$ such that $\gamma_j>0$
and $\zeta(\tfrac12+i\gamma_j)=0$ for each $j$, and $\obg\le T$.
\end{definition*}


Let $W$, $A$, $B$, $C$, and $D$ be defined as in \S\ref{sec:underlying}.
Replacing $t$ by $t-\widetilde\gamma$ in \eqref{eq:mont-orig}
and taking into account the relation
\be\label{eq:carr}
W_{\gamma}(t-\widetilde\gamma)=W_{\eobg}(t)
=\frac{2}{(t-\obg)^2+1}
\qquad\text{for all}\quad\bg=(\widetilde\gamma,\gamma),
\ee
we get that
\be\label{eq:precog}
\mathop{\sum_{\gamma}}
x^{i\gamma}W_{\eobg}(t)
=A(t-\widetilde\gamma)+B(t-\widetilde\gamma)
+C(t-\widetilde\gamma)+D(t-\widetilde\gamma).
\ee
Let $\ind{\widetilde\gamma}$ denote the indicator function of
the interval $[\widetilde\gamma,\infty)$.
Multiplying \eqref{eq:precog} by
$x^{i\widetilde\gamma}\ind{\widetilde\gamma}(t)$,
summing over all $\widetilde\gamma\in(0,T]$ (counted with multiplicity),
and integrating over $[0,T]$,
we derive the basic identity
\be\label{eq:basic-identity2}
\begin{split}
&\int_0^T\Bigg|
\mathop{\sum_{\widetilde\gamma\in(0,T]}\sum_{\gamma}}
x^{i\eobg}\ind{\widetilde\gamma}(t)W_{\eobg}(t)\Bigg|^2\dd t\\
&\qquad=\int_0^T\Biggl|\sum_{\widetilde\gamma\in(0,T]}
x^{i\widetilde\gamma}\ind{\widetilde\gamma}(t)
\bigr\{A(t-\widetilde\gamma)+B(t-\widetilde\gamma)
+C(t-\widetilde\gamma)+D(t-\widetilde\gamma)\bigr\}\Biggr|^2\dd t,
\end{split}
\ee
where $\bg=(\widetilde\gamma,\gamma)$.
The indicator $\ind{\widetilde\gamma}$ confines each
$\widetilde\gamma$-summand to the range $t\ge\widetilde\gamma$,
keeping the arguments of $A$, $B$, $C$, $D$
nonnegative and matching the constraint $\obg\le T$ that
defines $\fZ_2(T)$. Following Montgomery \cite{Mont},
our approach is to first estimate the integral on
the left side of \eqref{eq:basic-identity2}
in terms of the quantity
\[
G_2(x,T)\defeq
3\Bigl(\frac{2\pi}{T\sLT}\Bigr)^{3}
\ssum{\bg,\bgp\in\fZ_2(T)}
x^{i(\eobg-\eobgp)}
w(\obg-\obgp),
\]
noting that $G_2(\alpha)=G_2(T^\alpha,T)$,
and to then estimate the right side of \eqref{eq:basic-identity2}
using individual estimates of the four integrals
corresponding to $A$, $B$, $C$, and $D$.

\subsection{Setup}\label{sec:notation2}
For any function $\Phi\colon[0,T]\to\C$, we write
\[
M[\Phi]\defeq\int_0^T\bigl|\Phi(t)\bigr|^2\dd t.
\]
The left side of \eqref{eq:basic-identity2} is $M_W\defeq M[S_W]$, where
\[
S_W(t)\defeq\mathop{\sum_{\widetilde\gamma\in(0,T]}\sum_\gamma}
x^{i\eobg}\ind{\widetilde\gamma}(t)W_{\eobg}(t).
\]
For $F=A$, $B$, $C$, or $D$ (see \S\ref{sec:underlying}),
we denote $M_F\defeq M[S_F]$, where
\[
S_F(t)\defeq\sum_{\widetilde\gamma\in(0,T]}
x^{i\widetilde\gamma}\ind{\widetilde\gamma}(t)\,F(t-\widetilde\gamma).
\]
Then \eqref{eq:basic-identity2} is precisely the identity
\[
M_W=M[S_A+S_B+S_C+S_D].
\]
We estimate $M_W$, $M_A$, $M_B$, $M_C$, and $M_D$
individually, and combine the results in \S\ref{sec:endgame2}.

\subsection{W-integral}\label{sec:W2}
By the definitions of \S\ref{sec:notation2},
\be\label{eq:MWf-defn-2}
M_W=\sum_{\widetilde\gamma,\widetilde\gamma'\in(0,T]}\sum_{\gamma,\gamma'}
x^{i(\eobg-\eobgp)}
\int_{\max\{\widetilde\gamma,\widetilde\gamma'\}}^T
W_{\eobg}(t)W_{\eobgp}(t)\dd t,
\ee
where $\bg=(\widetilde\gamma,\gamma)$ and
$\bgp=(\widetilde\gamma',\gamma')$.
In view of \eqref{eq:carr}, the last integral can be written as
\[
\int_{\max\{\widetilde\gamma,\widetilde\gamma'\}}^T
\frac{2}{(t-\widetilde\gamma-\gamma)^2+1}\cdot
\frac{2}{(t-\widetilde\gamma'-\gamma')^2+1}\dd t.
\]
Applying Lemma~\ref{lem:lime} with $a\defeq\widetilde\gamma$ and
$b\defeq\widetilde\gamma'$,
it follows that
\dalign{
M_W&=\sum_{\widetilde\gamma,\widetilde\gamma'\in(0,T]}
\ssum{\gamma\in(0,T-\widetilde\gamma]\\\gamma'\in(0,T-\widetilde\gamma']}
x^{i(\eobg-\eobgp)}
\int_\R W_{\eobg}(t)W_{\eobgp}(t)\dd t+O\bigl(N(T)^2\sLT^3\bigr)\\
&=\sum_{\bg,\bgp\in\fZ_2(T)}
x^{i(\eobg-\eobgp)}
\int_\R W_{\eobg}(t)W_{\eobgp}(t)\dd t+O(T^{2}\sLT^{5}).
}
The last integral is evaluated by the following lemma.

\begin{lemma}\label{lem:WW}
For any $a,b\in\R$, we have
\[
\int_\R W_a(t)W_b(t)\dd t=2\pi\cdot w(a-b),
\]
where $W_c(t)\defeq 2/((t-c)^2+1)$ and $w$ is given by
\eqref{eq:w(u)defn}.
\end{lemma}

\begin{proof}
By a residue computation (or by the definition \eqref{eq:w(u)defn} of
$w$), for any $c\in\R$ the function $W_c(t)=2/((t-c)^2+1)$ has the
Fourier transform
\[
\widehat{W_c}(\xi)\defeq\int_\R W_c(t)\er^{-i\xi t}\dd t
=2\pi\,\er^{-ic\xi-|\xi|}\qquad(\xi\in\R).
\]
Since $W_a,W_b\in L^2(\R)$, Plancherel's theorem and the identity
$w(v)=\int_\R\er^{-2|\xi|}\er^{iv\xi}\dd\xi$ give
\[
\int_\R W_a(t)W_b(t)\dd t
=\frac{1}{2\pi}\int_\R\widehat{W_a}(\xi)\,\overline{\widehat{W_b}(\xi)}\dd\xi
=2\pi\int_\R\er^{-2|\xi|}\er^{-i(a-b)\xi}\dd\xi
=2\pi\cdot w(a-b),
\]
as claimed.
\end{proof}

Taking $a\defeq\obg$ and $b\defeq\obgp$
in Lemma~\ref{lem:WW}, we conclude that
\be\label{eq:ident-W-two}
\begin{split}
M_W&=2\pi\sum_{\bg,\bgp\in\fZ_2(T)}
x^{i(\eobg-\eobgp)}
w(\obg-\obgp)+O(T^{2}\sLT^{5})\\
&=\frac{2\pi}{3}\Bigl(\frac{T\sLT}{2\pi}\Bigr)^{3}G_2(x,T)
+O(T^{2}\sLT^{5}).
\end{split}
\ee

\subsection{A-integral}\label{sec:A2}
By the definitions of \S\ref{sec:notation2},
\[
M_A=\sum_{\widetilde\gamma,\widetilde\gamma'\in(0,T]}
x^{i(\widetilde\gamma-\widetilde\gamma')}
\int_{\max\{\widetilde\gamma,\widetilde\gamma'\}}^T
A(t-\widetilde\gamma)\overline{A(t-\widetilde\gamma')}\dd t,
\]
where (cf.~\S\ref{sec:underlying})
\[
A(t)\defeq-\frac{\zeta'}{\zeta}(-\tfrac12+it)x^{-1+it}=f(t)x^{-1+it}
\]
with
\[
f(t)\defeq\frac{\zeta'}{\zeta}(\tfrac32-it)
-\log\pi+\frac12\frac{\Gamma'}{\Gamma}\Bigl(\frac{-\tfrac12+it}{2}\Bigr)
+\frac12\frac{\Gamma'}{\Gamma}\Bigl(\frac{\tfrac32-it}{2}\Bigr),
\]
in view of the functional equation of the
logarithmic derivative of the zeta function:
\[
\frac{\zeta'}{\zeta}(s)
+\frac{\zeta'}{\zeta}(1-s)
=\log\pi-\frac12\frac{\Gamma'}{\Gamma}\Bigl(\frac{s}{2}\Bigr)
-\frac12\frac{\Gamma'}{\Gamma}\Bigl(\frac{1-s}{2}\Bigr).
\]
Using standard estimates for the digamma function
(see, e.g., \cite[App.~C]{MontVauBook}) together
with the bound $\frac{\zeta'}{\zeta}(\tfrac32-it)\ll 1$, 
we have
\be\label{eq:dietcoke}
f(t)=\log\tau_t+O(1),\qquad\tau_t\defeq|t|+10.
\ee
Since
\[
x^{i(\widetilde\gamma-\widetilde\gamma')}
A(t-\widetilde\gamma)\overline{A(t-\widetilde\gamma')}
=x^{-2}f(t-\widetilde\gamma)\overline{f(t-\widetilde\gamma')},
\]
it follows that
\[
M_A=x^{-2}\int_0^T\biggl|\sum_{\widetilde\gamma\in(0,t]}
f(t-\widetilde\gamma)\biggr|^2\dd t.
\]
By \eqref{eq:dietcoke}, \eqref{eq:N(u)estimate}, and
Lemma~\ref{lem:specialized} with $T$ replaced by $t$
(so that $\sLTtau=\log(t+10)$ at the point of evaluation),
we find that
\[
\sum_{\widetilde\gamma\in(0,t]}f(t-\widetilde\gamma)
=\frac{t(\log t)^2}{2\pi}+O(t\log t\,\log\sLT)
\qquad(14\le t\le T).
\]
Moreover, for $t\in[0,14)$ the sum is empty since $\gamma_1>14.13$.
Inserting these results into the previous expression for $M_A$,
a short calculation leads to 
\[
M_A=\frac{T^{3}\sLT^{4}}
{12\pi^2\,x^2}
\biggl\{1+O\(\frac{\log\sLT}{\sLT}\)\biggr\}.
\]

\subsection{B-integral}\label{sec:B2}
By the definitions of \S\ref{sec:notation2},
\[
M_B=\int_0^T\bigl|S_B(t)\bigr|^2\dd t,
\]
and (by the definition of $B$ in \S\ref{sec:underlying})
\dalign{
S_B(t)&=\sum_{\widetilde\gamma\in(0,T]}x^{i\widetilde\gamma}
\ind{\widetilde\gamma}(t)\,B(t-\widetilde\gamma)\\
&=-x^{-1/2}\sum_{\widetilde\gamma\in(0,T]}
\ind{\widetilde\gamma}(t)\biggl\{
\sum_{n\le x}\Lambda(n)n^{i\widetilde\gamma}\(\frac{x}{n}\)^{-1/2+it}
+\sum_{n>x}\Lambda(n)n^{i\widetilde\gamma}\(\frac{x}{n}\)^{3/2+it}\biggr\}\\
&=-x^{-1/2}\biggl\{
\sum_{n\le x}\Lambda(n)\,S(n,t)\(\frac{x}{n}\)^{-1/2+it}
+\sum_{n>x}\Lambda(n)\,S(n,t)\(\frac{x}{n}\)^{3/2+it}\biggr\},
}
where (as in \eqref{eq:STLrocks})
\[
S(n,t)\defeq\sum_{\gamma\in(0,t]}n^{i\gamma}=\lambda_nt+E(n,t).
\]

To evaluate $M_B$, we first separate the main term of $S_B$
from its error, and then appeal to a mean value theorem. Thus,
we now write
\be\label{eq:SB-split}
S_B(t)=-x^{-1/2}\bigl\{t\,P(t)+Q(t)\bigr\},
\ee
where
\dalign{
P(t)&\defeq\sum_{n\le x}\Lambda(n)\lambda_n\(\frac{x}{n}\)^{-1/2+it}
+\sum_{n>x}\Lambda(n)\lambda_n\(\frac{x}{n}\)^{3/2+it},\\
Q(t)&\defeq\sum_{n\le x}\Lambda(n)E(n,t)\(\frac{x}{n}\)^{-1/2+it}
+\sum_{n>x}\Lambda(n)E(n,t)\(\frac{x}{n}\)^{3/2+it}.
}
Since $|x^{it}|=1$, we have $|P(t)|=\bigl|\sum_nb_nn^{-it}\bigr|$ with
\[
b_n\defeq-\frac{\Lambda(n)^2}{2\pi x^{1/2}}\quad(n\le x),\qquad
b_n\defeq-\frac{\Lambda(n)^2x^{3/2}}{2\pi n^{2}}\quad(n>x).
\]
Noting that $\sum_nb_n^2\,n<\infty$,
we may apply \cite[Cor.\,3]{MontVau} to deduce that
\[
I(t)\defeq\int_0^t|P(u)|^2\dd u=\sum_nb_n^2(t+O(n)).
\]
Integrating by parts, we have
\[
\int_0^Tt^2|P(t)|^2\dd t=T^2I(T)-2\int_0^TtI(t)\dd t
=\sum_nb_n^2\bigl\{\tfrac13T^3+O(nT^2)\bigr\}.
\]
By Lemma~\ref{lem:wagyu-sirloin},
\[
\sum_{n\le x}b_n^2=\frac{\sLx^3}{4\pi^2}\{1+O(\sLx^{-1})\},\qquad
\sum_{n>x}b_n^2=\frac{\sLx^3}{12\pi^2}\{1+O(\sLx^{-1})\},
\]
and also $\sum_nb_n^2\,n\ll x\sLx^3$; consequently,
\be\label{eq:P-main}
x^{-1}\int_0^Tt^2|P(t)|^2\dd t
=\frac{\sLx^3T^3}{9\pi^2x}\{1+O(\sLx^{-1})\}+O(\sLx^3T^2).
\ee

It remains to estimate $Q$.
Uniformly for $0\le t\le T$ we have
\be\label{eq:E-pointwise}
E(n,t)\ll\begin{cases}
n^{1/2}\,\sLT\log\sLT
&\quad\hbox{if $2\le n\le T^2$},\\
T\sLT
&\quad\hbox{if $n>T^2$}.\\
\end{cases}
\ee
Indeed, for $14\le t\le T$ and $n\le T^2$ this follows from
\eqref{eq:lam-cE-defn}, since $\log(nt)\ll\sLT$ and
$\log\log(nt)\ll\log\sLT$ in that range, whereas for $0\le t<14$
one has $S(n,t)=0$ since $\gamma_1>14.13$,
and $E(n,t)=-\lambda_nt\ll 1$. For $n>T^2$ we use instead
$|S(n,t)|\le N(T)\ll T\sLT$ and $\lambda_nt\ll T$.
By \eqref{eq:E-pointwise} and Lemma~\ref{lem:wagyu-sirloin},
uniformly for $0\le t\le T$, we have
\dalign{
Q(t)&\ll\frac{\sLT\log\sLT}{x^{1/2}}\ssum{n\le x}\Lambda(n)\,n
+x^{3/2}\,\sLT\log\sLT\ssum{x<n\le T^2}\frac{\Lambda(n)}{n}
+x^{3/2}\,T\sLT\ssum{n>T^2}\frac{\Lambda(n)}{n^{3/2}}\\
&\ll x^{3/2}\,\sLT^2\log\sLT,
}
and therefore
\be\label{eq:Q-bound}
\int_0^T|Q(t)|^2\dd t\ll x^3\,T\,\sLT^4\log^2\sLT.
\ee
Expanding \eqref{eq:SB-split} and applying the Cauchy--Schwarz
inequality to the cross term, \eqref{eq:P-main} and
\eqref{eq:Q-bound} yield the estimate
\[
M_B=x^{-1}\int_0^Tt^2|P(t)|^2\dd t
+O\bigl(x^{1/2}T^2\,\sLx^{3/2}\sLT^2\log\sLT
+x^2T\,\sLT^4\log^2\sLT\bigr),
\]
and both error terms are $\ll x^{-1}\sLx^2T^3$ for
$x\le T^{2/3}\sLT^{-2}$. Hence,
\be\label{eq:MB-final}
M_B=\frac{\sLx^3 T^3}{9\pi^2x}+O\bigl(x^{-1}\sLx^2T^3\bigr)
\qquad\bigl(16\le x\le T^{2/3}\sLT^{-2}\bigr).
\ee

\subsection{C- and D-integrals}\label{sec:C2+D2}

Here (cf.~\S\ref{sec:underlying})
\dalign{
C(t)\defeq\frac{x^{1/2}}{\tfrac32-it}+\frac{x^{1/2}}{\tfrac12+it},\qquad
D(t)\defeq\sum_{n\in\N}\frac{x^{-2n-1/2}}{2n-\tfrac12+it}
-\sum_{n\in\N}\frac{x^{-2n-1/2}}{2n+\tfrac32+it}.
}
For all real $t$ we have $|C(t)|\ll x^{1/2}\tau_t^{-1}$, and also
$|D(t)|\ll x^{-1/2}\tau_t^{-1}$, since the $n$th summand in $D(t)$ is
$\ll x^{-2n-1/2}(n^2+t^2)^{-1}$ and
$\sum_{n\in\N}(n^2+t^2)^{-1}\ll\tau_t^{-1}$.
Moreover, grouping the ordinates into unit intervals and using the
bound $N(u+1)-N(u)\ll\sLT$ $(0\le u\le T)$, immediate from
\eqref{eq:N(u)estimate}, we have
\be\label{eq:tau-sum}
\sum_{\widetilde\gamma\in(0,T]}\frac{1}{\tau_{t-\widetilde\gamma}}
\ll\sLT\sum_{0\le k\le T}\frac{1}{|t-k|+10}
\ll\sLT^2\qquad(0\le t\le T).
\ee
By the triangle inequality, followed by $\ind{\widetilde\gamma}(t)\le 1$,
for $0\le t\le T$ we have
\[
S_C(t)\defeq\sum_{\widetilde\gamma\in(0,T]}
x^{i\widetilde\gamma}\ind{\widetilde\gamma}(t)\,C(t-\widetilde\gamma).
\]
\[
|S_C(t)|=\bigg|\sum_{\widetilde\gamma\in(0,T]}
x^{i\widetilde\gamma}\ind{\widetilde\gamma}(t)
\,C(t-\widetilde\gamma)\bigg|
\le\sum_{\widetilde\gamma\in(0,T]}\bigl|C(t-\widetilde\gamma)\bigr|
\ll x^{1/2}\sLT^2,
\]
and likewise $|S_D(t)|\ll x^{-1/2}\sLT^2$; therefore
\[
M_C\ll xT\sLT^{4}
\mand
M_D\ll x^{-1}T\sLT^{4}.
\]

\subsection{Endgame}\label{sec:endgame2}

\begin{lemma}\label{lem:cross}
With $S_A$, $S_B$ as in \S\ref{sec:notation2}, uniformly for
$16\le x\le T^{2/3}\sLT^{-2}$, we have
\[
\int_0^T S_A(t)\overline{S_B(t)}\dd t
\ll\frac{M_A+M_B}{\log\sLT}.
\]
\end{lemma}

\begin{proof}
Suppose $16\le x\le T^{2/3}\sLT^{-2}$, and let
$P$, $Q$, $\lambda_n$, and $E(n,t)$ be defined as in \S\ref{sec:B2}. 
From \S\ref{sec:A2},
\[
x^{i\widetilde\gamma}A(t-\widetilde\gamma)
=x^{-1+it}f(t-\widetilde\gamma)
\]
with
\dalign{
f(t)&\defeq
-\log\pi+\frac12\frac{\Gamma'}{\Gamma}\Bigl(\frac{-\tfrac12+it}{2}\Bigr)
+\frac12\frac{\Gamma'}{\Gamma}\Bigl(\frac{\tfrac32-it}{2}\Bigr)
+\frac{\zeta'}{\zeta}(\tfrac32-it)
=\phi(t)-\sum_{m\in\N}\frac{\Lambda(m)}{m^{3/2-it}},
}
where $\phi(u)\ll\log\tau_u$ and $\phi'(u)\ll\tau_u^{-1}$
\cite[App.~C]{MontVauBook}. Then
\[
S_A(t)\defeq\sum_{\widetilde\gamma\in(0,T]}
x^{i\widetilde\gamma}\ind{\widetilde\gamma}(t)
x^{-1+it}\bigg\{\phi(t-\widetilde\gamma)
-\sum_{m\in\N}\frac{\Lambda(m)}{m^{3/2-i(t-\widetilde\gamma)}}\bigg\}
=g_1(t)+g_2(t),
\]
where
\[
g_1(t)\defeq\ssum{\widetilde\gamma\in(0,t]}\phi(t-\widetilde\gamma),
\qquad
g_2(t)\defeq-\ssum{m\ge2}\frac{\Lambda(m)}{m^{3/2}}\,
m^{it}\,\overline{S(m,t)}.
\]
and so by
\eqref{eq:SB-split},
\be\label{eq:X-reduce}
\int_0^TS_A(t)\overline{S_B(t)}\dd t
=-x^{-3/2}\int_0^Tx^{it}\bigl\{g_1(t)+g_2(t)\bigr\}
\bigl\{t\,\overline{P(t)}+\overline{Q(t)}\bigr\}\dd t.
\ee

We now record four facts, each uniform for $0\le t\le T$.

First, since
$\tau_{t-\widetilde\gamma}\le\tau_t$ for $\widetilde\gamma\in(0,t]$ and
$N(t)\ll t\sLT$ (see \eqref{eq:N(u)estimate}), we have
\be\label{eq:g0-sup}
g_1(t)\ll N(t)\log\tau_t\ll t\,\sLT^2,
\qquad\text{and thus}\quad
\int_0^T|g_1(t)|\dd t\ll T^2\sLT^2.
\ee

Second, $g_1$ is piecewise $C^1$, with an upward jump of size
$|\phi(0)|\ll1$ at each ordinate in $(0,T]$ and derivative
$g_1'(t)=\sum_{\widetilde\gamma\le t}\phi'(t-\widetilde\gamma)\ll\sLT^2$
between jumps, by \eqref{eq:tau-sum}. Writing $V(h)$ for the total
variation of $h$ on $[0,T]$, the jumps contribute $\ll N(T)\ll T\sLT$
and the smooth part $\int_0^T|g_1'|\ll T\sLT^2$, so that
\be\label{eq:g0-var}
V(g_1)\ll T\sLT^2,\qquad
V\bigl(t\,g_1(t)\bigr)
\le T\,V(g_1)+T\sup_{[0,T]}|g_1|\ll T^2\sLT^2,
\ee
the second bound by the product rule for total variation together
with \eqref{eq:g0-sup}.

Third, by \eqref{eq:STLrocks}, \eqref{eq:E-pointwise}, and
Lemma~\ref{lem:wagyu-sirloin},
\[
g_2(t)\ll t\ssum{m\ge2}\frac{\Lambda(m)^2}{m^2}
+\sLT\log\sLT\ssum{2\le m\le T^2}\frac{\Lambda(m)}{m}
+T\sLT\ssum{m>T^2}\frac{\Lambda(m)}{m^{3/2}}
\ll t+\sLT^2\log\sLT,
\]
whence $\int_0^T|g_2(t)|^2\dd t\ll T^3$. 

Fourth, if $h$ has bounded
variation on $[0,T]$ with $h(0)=0$, then for $n\ge2$, integration by
parts gives
\be\label{eq:vdc}
\int_0^Th(t)\,n^{it}\dd t
=\frac{1}{i\log n}\biggl\{\bigl[h(t)\,n^{it}\bigr]_0^T
-\int_0^Tn^{it}\dd h(t)\biggr\}
\ll\frac{V(h)}{\log n}.
\ee

We now split \eqref{eq:X-reduce} into four terms $X_{1P}$, $X_{1Q}$,
$X_{2P}$, $X_{2Q}$, according to $g=g_1+g_2$ and
$t\overline P+\overline Q$.
For the $g_1$-terms, we expand $P$ and $Q$ over $n$ and use
$x^{it}\,\overline{(x/n)^{it}}=n^{it}$, the interchange justified by
absolute convergence, writing $c_n\defeq(n/x)^{1/2}$ for $n\le x$ and
$c_n\defeq(x/n)^{3/2}$ for $n>x$. Applying \eqref{eq:vdc} with
$h(t)=t\,g_1(t)$, then \eqref{eq:g0-var}, the inequality
$\Lambda(n)^2/\log n\le\Lambda(n)$, and Lemma~\ref{lem:wagyu-sirloin},
\dalign{
X_{1P}&\ll x^{-3/2}\,T^2\sLT^2
\ssum{n\ge2}\frac{\Lambda(n)^2c_n}{n^{1/2}\log n}\\
&\ll x^{-3/2}\,T^2\sLT^2\biggl(x^{-1/2}\ssum{n\le x}\Lambda(n)
+x^{3/2}\ssum{n>x}\frac{\Lambda(n)}{n^2}\biggr)
\ll\frac{T^2\sLT^2}{x}.
}
As $x\log\sLT\ll T\sLT^2$ throughout the range, the bound
$X_{1P}\ll M_A/\log\sLT$ holds by \S\ref{sec:A2}. 
Similarly, by \eqref{eq:g0-sup}, \eqref{eq:E-pointwise},
and Lemma~\ref{lem:wagyu-sirloin}, $X_{1Q}$ is
\dalign{
&\ll x^{-3/2}\,T^2\sLT^2\biggl\{\sLT\log\sLT
\biggl(x^{-1/2}\ssum{n\le x}\Lambda(n)\,n
+x^{3/2}\ssum{x<n\le T^2}\frac{\Lambda(n)}{n}\biggr)
+T\sLT\,x^{3/2}\ssum{n>T^2}\frac{\Lambda(n)}{n^{3/2}}\biggr\}\\
&\ll T^2\sLT^4\log\sLT.
}
When $x\le T^{1/2}\log^{-2}\sLT$ this gives $X_{1Q}\ll M_A/\log\sLT$,
since $x^2\log^2\sLT\ll T$; otherwise $\sLx\ge\tfrac13\sLT$, so
\eqref{eq:MB-final} yields $M_B\asymp\sLx^3T^3/x$ in this case;
since $x\,\sLT\log^2\sLT\ll T$, we obtain that $X_{1Q}\ll M_B/\log\sLT$.

For the $g_2$-terms, we apply Cauchy--Schwarz, using
\[
\int_0^T|g_2|^2\ll T^3
\mand
\int_0^Tt^2|P(t)|^2\dd t\ll\sLx^3T^3
\quad\text{(from \eqref{eq:P-main}, since $x\le T$)},
\]
together with \eqref{eq:Q-bound}, we derive that
\[
X_{2P}\ll x^{-3/2}\,T^3\sLx^{3/2},
\qquad
X_{2Q}\ll x^{-3/2}\,T^{3/2}
\Bigl(\int_0^T|Q(t)|^2\dd t\Bigr)^{1/2}
\ll T^2\sLT^2\log\sLT.
\]
If $x^{1/2}\sLx^{3/2}\le\sLT^4/\log\sLT$, then $X_{2P}\ll M_A/\log\sLT$;
otherwise $x\ge\sLT^4$ for $T$ large, so $\sLx\gg\log\sLT$,
\eqref{eq:MB-final} gives $M_B\asymp\sLx^3T^3/x$, and
$\log\sLT\ll x^{1/2}\sLx^{3/2}$; together, these facts yield the
bound $X_{2P}\ll M_B/\log\sLT$. The term
$X_{2Q}$ is dominated by the bound for $X_{1Q}$ and treated identically.
Combining the four estimates completes the proof.
\end{proof}

\begin{proof}[Proof of Theorem~\ref{thm:main}]
Recall from \S\ref{sec:notation2}:
\[
M_W=M[S_A+S_B+S_C+S_D].
\]
By the results of \S\S\ref{sec:A2}$-$\ref{sec:C2+D2}, for
$16\le x\le T^{2/3}\sLT^{-2}$ we have
\dalign{
M_A&=\frac{T^3\sLT^4}{12\pi^2x^2}
\bigl\{1+O(\sLT^{-1}\log\sLT)\bigr\},\qquad
&M_B&=\frac{\sLx^3 T^3}{9\pi^2x}\bigl\{1+O(\sLx^{-1})\bigr\},\\
M_C&\ll xT\sLT^4,\qquad
&M_D&\ll x^{-1}T\sLT^4,
}
so that $M_A,M_B\ge M_C,M_D$ and, throughout this range,
\[
\sqrt{(M_A+M_B)(M_C+M_D)}\ll\frac{M_A+M_B}{\log\sLT}.
\]
Applying Lemma~\ref{lem:C-S} with $F_1=S_A$, $F_2=S_B$, $F_3=S_C$,
$F_4=S_D$, and invoking Lemma~\ref{lem:cross} for the cross term
$2\Re\int_0^TS_A(t)\overline{S_B(t)}\dd t$, we obtain
\[
M_W=M_A+M_B+O\(\frac{M_A+M_B}{\log\sLT}\)
=\biggl\{\frac{T^3\sLT^4}{12\pi^2x^2}
+\frac{\sLx^3 T^3}{9\pi^2x}\biggr\}
\bigl\{1+O((\log\sLT)^{-1})\bigr\}.
\]
For $1\le x<16$ the same formula holds, since trivial bounds give
$M_B\ll T^3\sLT^2$ and $\sLx^3T^3/x\ll T^3\ll M_A/\sLT$, so all
terms other than $M_A$ are absorbed in the error. On the other hand,
recalling \eqref{eq:basic-identity2} and the definition
\eqref{eq:MWf-defn-2}, we have by \eqref{eq:ident-W-two}
\[
G_2(x,T)=\biggl\{\frac{\sLT}{x^2}
+\frac{4\sLx^3}{3x\sLT^3}\biggr\}
\bigr\{1+O((\log\sLT)^{-1})\bigl\}
\qquad(x\le T^{2/3}\sLT^{-2}).
\]
Taking $x\defeq T^\alpha$, we complete the proof of
Theorem~\ref{thm:main}.
\end{proof}

\end{document}